\documentclass{svmult}

\usepackage{mathptmx}
\usepackage{helvet}
\usepackage{courier}
\usepackage{amsmath,amsfonts}
\usepackage{graphicx}
\usepackage[bottom]{footmisc}
\usepackage[numbers]{natbib}    
\usepackage{color}
\usepackage{multirow}
\usepackage{url}

\newcommand{\f}[1]{\mathbf{#1}}
\newcommand{\RR}{\mathbb{R}}
\newcommand{\CC}{\mathbb{C}}

\begin{document}

\title*{Fast Approximate Implicitization of Envelope Curves using Chebyshev
Polynomials}
\author{Oliver J. D. Barrowclough \and Bert J\"uttler \and Tino Schulz }
\institute{%
  Oliver J. D. Barrowclough \at
  SINTEF ICT, Applied Mathematics, P.O. Box 124, Blindern, 0314 Oslo, Norway
  \email{oliver.barrowclough@sintef.no}
  \\
  Bert J\"uttler \at
  Johannes Kepler University, Institute of Applied Geometry, Altenberger Str.
69, A-4040 Linz, Austria 
  \email{bert.juettler@jku.at}
  \\
  Tino Schulz \at
  GALAAD, INRIA M\'editerran\'ee, 2004 Route des Lucioles, 06902
Sophia-Antipolis, France
  \email{tino.schulz@inria.fr}}

 \titlerunning{Fast Envelope Implicitization}
 \authorrunning{Barrowclough et al.}

\maketitle

\abstract{Consider a rational family of planar rational curves in a
  certain region of interest.  We are interested in finding an
  approximation to the implicit representation of the envelope.  Since
  exact implicitization methods tend to be very costly, we employ an
  adaptation of approximate implicitization to envelope computation.
  Moreover, by utilizing an orthogonal basis in the construction
  process, the computational times can be shortened and the numerical
  condition improved. We provide an example to illustrate the
  performance of our approach.}



\section{Introduction}
\let\thefootnote\relax\footnote{The final publication is available at Springer via:\\ \url{http://dx.doi.org/ 10.1007/978-94-007-4620-6_26}}
In geometric applications there are two basic standards for
representing curves, namely the {\em parametric} and the {\em
  implicit} descriptions. Both descriptions feature specific advantages
and disadvantages that complement each other. For instance, parametric 
curves allow the simple generation of point samples, while implicit 
forms support the decision of point location queries. In many applications, 
such as intersection computations, it is an advantage if both representations 
are available, and conversion algorithms are therefore of substantial practical 
interest. The conversion processes are called {\em parametrization} and 
{\em implicitization}, respectively.

A rational curve may always be implicitized, whereas the opposite is
not true \cite{sendra:rational}.  Several techniques for {\em exact}
implicitization exist, e.g. Gr\"obner bases, moving curves/surfaces,
or methods based on resultants, (see e.g.
\cite{hoffmann:implicit}).  However, due to their
computational complexity, their practical use is often restricted to
low-degree curves.  Moreover, the variety obtained by exact
implicitization may contain unexpected branches and self-intersections.

A valid alternative to exact methods is {\em approximate
  implicitization}; cf.
\cite{dokken:approximate,dokken:weak}.  Instead of the exact
variety, a low degree approximation is used to represent the shape of
the geometric object in a certain region of interest.  This technique
can be implemented using floating point numbers and thus it offers
faster computation, while having very high convergence rates.  As
shown in \cite{barrowclough:linalg}, the speed-up may be increased
even further by using an orthogonal basis in the construction process.

These features make approximate implicitization a promising candidate
for an efficient computation of {\em envelope curves}.  Envelopes are
used in different contexts in mathematics and applications, ranging
from gearing theory and geometrical optics to NC-machining and 
Computer-Aided Design. In robotics, envelopes are ubiquitous,
appearing for instance as singularities or boundaries. The theory of envelopes 
is covered by the classical literature, 
and is continuously extended, due to their practical importance
\cite{abdel:swept,kim:fast,peternell:swept,jue:exact}.

Approximate implicitization has recently been adapted to the
computation of envelopes in \cite{schulz:envelope}. As shown there,
the idea is feasible and most properties of the original method can be
preserved, such as the possibility of obtaining the exact
solution. However, the convergence behaviour for higher degrees has
not previously been studied and the computations are still fairly expensive,
needing integrals of products of high degree polynomials.

The present paper uses the latest results from approximate
implicitization to obtain a fast and efficient algorithm for 
approximating the envelope. This will make the use of
implicit methods more attractive and moreover allow us to study the
convergence behaviour experimentally.
The paper is organized as follows: In Section two we will recall the
basics of envelopes of planar curves.  After that, the third section
shows how approximate implicitization can be used to compute envelopes
and derives a fast and efficient algorithm.  The performance of our
approach is illustrated with an example and discussion in Section
four.

\section{Envelopes of Rational Families of  Curves}
Consider the family of rational curves
\begin{equation}\label{eq:92}
\f{p}(s,t)= \bigl( x(s,t)/w(s,t),y(s,t)/w(s,t) \bigr)^T, \quad
(s,t)\in I\times J
\end{equation}
where $x$, $y$ and $w$ are bivariate polynomials of bidegree $(n_1,n_2)$
with $\gcd(x,y,w)=1$ and $I, J\subset\RR$ are closed intervals.  We
assume that $w(s,t)\neq 0$ for all ${(s,t)\in I\times J}$. Either $s$ or
$t$ can be thought of as the time-like parameters, and the remaining
parameter $t$ or $s$ is then used to parameterize the curves forming
the family.

The {\em envelope} of the mapping $\f p$ consists of those points
where its Jacobian $\f J(s,t)$ becomes singular.  We observe that
$\det \f J(s,t) = h(s,t)/w(s,t)^3$, where
\begin{equation}
h(s,t)=\det\left( \begin{array}{ccc} x(s,t) 
& \partial_s x(s,t) & \partial_t x(s,t) \\
y(s,t) & \partial_s y(s,t) & \partial_t y(s,t) \\
w(s,t) & \partial_s w(s,t) & \partial_t w(s,t) \end{array} \right).
\end{equation}
The function $h$ is called the {\em envelope function}, since its zero
set determines those points in the parameter domain which are mapped to
the envelope. Unfortunately, certain parts of the zero set of $h$ may
degenerate under the mapping $\f p$.  

The earlier paper \cite{schulz:envelope} describes how these {\em
  improper factors} can be removed from $h$.  This can be done via
some gcd computation and gives the {\em reduced envelope function
  $\tilde h$}.  The image of the zero set of $\tilde h$ under $\f p$
is called {\em proper part} of the envelope.

Let $q:\CC^2 \rightarrow \CC$ be the polynomial which defines {\em 
the implicit equation of the proper part of the envelope} of $\f p$.
According to Theorem 1 of \cite{schulz:envelope}, there exists a real
polynomial $\lambda(s,t)$ such that
\begin{equation}\label{eq:148}
\left( q\circ \f p \right)(s,t)w(s,t)^d = \lambda (s,t) \tilde h (s,t)^2.
\end{equation}
Equation \eqref{eq:148} is linear with respect to the coefficients of
$q$ and $\lambda$. Let
\begin{equation}\label{eq:bases}
q(\mathbf{x}) = \mathbf{c}_q^T\beta(\mathbf{x}) \quad \text{and} \quad
\lambda(s,t) = \mathbf{c}_{\lambda}^T\alpha(s,t),
\end{equation}
where $\beta(\mathbf{x})=(\beta_k(\mathbf{x}))_{k=1}^M$ and 
$\alpha(s,t)=(\alpha_i(s)\alpha_j(t))_{(i,j)=(0,0)}^{(k_1,k_2)}$ are bases of
polynomials in $\f x$ and $s,t$ of total degree $d$ and bidegree $(k_1,k_2)$
respectively, where $M=\binom{d+2}{2}$. The coefficients of $q$ and $\lambda$
with respect to these bases form a vector $\mathbf{c} =
(\mathbf{c}_q^T,\mathbf{c}_\lambda^T)^T$.  We formulate the {\em
  problem of approximate envelope implicitization:} Find the
coefficients $\mathbf{c}$ which solve the weighted least squares
minimization problem
\begin{equation}\label{eq:min}
 \min_{\Vert\mathbf{c}\Vert_2=1} \int_{I\times J} \omega(s,t) \left( (q\circ
\f{p})(s,t) w(s,t)^d -\lambda(s,t) h(s,t)^2 \right)^2 \mathrm{d}(s,t),
\end{equation}
for a nonnegative weight function $\omega,$ and chosen degrees $d,$ $k_1$ and
$k_2.$ 
 
It is important to mention that we use $h$ instead of the exact
$\tilde h$, since our algorithm uses floating point computations which
do not support exact gcd computations.  While an exact solution of
this simplified problem might produce additional branches, the effect
on our low degree approximation will be negligible.

The result of the minimization \eqref{eq:min} depends both on the
choice of bases of $q$ and $\lambda$ and on the weight function
$\omega.$ The standard choice of a triangular Bernstein basis for $q$
and a tensor-product Bernstein basis for $\lambda$ has been used for
the approximations in this paper and also in
\cite{schulz:envelope}. However, as a major difference to the approach
in \cite{schulz:envelope} where $\omega\equiv1$, here we use a 
tensor product Chebyshev weight function on the domain $I\times J,$
for the reasons described in the next section.

\section{Fast Approximate Implicitization of Envelope
Curves}\label{section:fast}

The direct method for finding an approximate implicitization of envelope
curves by evaluating high degree integrals is simple, but 
computationally costly. In addition, the resulting symmetric positive
semi-definite matrix can be rather ill conditioned, leading to inaccurate 
null space computations when using floating point arithmetic. This is similar to
the case of approximate implicitization of parametric curves 
presented in \cite{barrowclough:linalg}. In that paper, an approach using
orthogonal polynomials is presented which greatly improves both the 
conditioning and the computation time of the problem. In this section we give the
details of how to implement the approach to approximate implicitization 
of envelope curves using Chebyshev polynomials.

\subsection{Approximate Implicitization using Chebyshev Polynomials}

As described previously, the method works by minimization of the integral
(\ref{eq:min}). Such a problem is aided by expressing the function 
in a basis orthonormal with respect to the chosen weight function $\omega.$ The
objective function is expressible in any tensor product polynomial 
basis of bidegree $$(L_1,L_2) =
(\max(dn_1,k_1+2\deg_s(h)),\max(dn_2,k_2+2\deg_t(h))).$$ Thus, choosing an
orthonormal basis (e.g., tensor-product Chebyshev 
polynomials), $\mathbf{T}(s,t) = (T_i(s)T_j(t))_{i=0,j=0}^{L_1,L_2}$ written in
vector form and using (\ref{eq:bases}), we can write
\begin{equation}\label{eq:chebyshev}
(q \circ \mathbf{p}) (s,t) w(s,t)^d - \lambda(s,t) h(s,t)^2 = \mathbf{T}(s,t)^T
(\mathbf{D}_q \mathbf{c}_q + \mathbf{D}_\lambda \mathbf{c}_\lambda),
\end{equation}
where the matrices $\mathbf{D}_q$ and $\mathbf{D}_\lambda$ contain coefficients
in $\mathbf{T}.$ Now, defining a matrix 
\begin{equation}\label{eq:dmatrix}
\mathbf{D} = (\mathbf{D}_q,\mathbf{D}_\lambda),
\end{equation}
we claim that the singular vector corresponding to the smallest singular value
of $\mathbf{D}$ solves the minimization problem (\ref{eq:min}). 
To see this, we prove the following Theorem:

\begin{theorem}
Let the matrix $\mathbf{D}$ be defined as in (\ref{eq:dmatrix}). Then we have 
\[
\min_{\Vert \mathbf{c} \Vert_2 = 1} \int_{I\times J} \omega
((q\circ\mathbf{p})w^d - \lambda h^2)^2 
= \min_{\Vert\mathbf{c}\Vert_2 = 1} \Vert \mathbf{D}\mathbf{c} \Vert_2^2.
\]
\end{theorem}
\begin{proof}
By (\ref{eq:chebyshev}) and (\ref{eq:dmatrix}) we have
\begin{eqnarray*}
       \int_{I\times J} \omega ((q\circ\mathbf{p})w^d - \lambda h^2 )^2    
= \int_{I\times J} \omega
(\mathbf{c}^T\mathbf{D}^T\mathbf{T})(\mathbf{T}^T\mathbf{D}\mathbf{c}) & = & \\
\mathbf{c}^T\mathbf{D}^T \left(\int_{I\times J} \omega
\mathbf{T}\mathbf{T}^T\right)\mathbf{D}\mathbf{c} 
= \mathbf{c}^T\mathbf{D}^T\mathbf{D}\mathbf{c} & = &
\Vert\mathbf{D}\mathbf{c}\Vert_2^2. \quad \qed
\end{eqnarray*}
\end{proof}
Since we have $\min_{\Vert \mathbf{c}\Vert_2=1}\Vert \mathbf{D}\mathbf{c}\Vert_2
= \sigma_{\min},$ where $\sigma_{\min}$ is the smallest 
singular value of $\mathbf{D},$ the corresponding right singular vector
solves the problem. The problem is, however, better conditioned 
and can be implemented in a more efficient way than the weak approach
\cite{barrowclough:linalg}.

\subsection{Implementation of the Chebyshev method}

The choice of using Chebyshev polynomials for the orthogonal basis is made
mainly for computational reasons; the coefficients 
can be generated via a fast algorithm. This utilizes an existing method outlined
for univariate polynomials in 
\cite{trefethen:spectral}, which exploits the discrete orthogonality of
Chebyshev polynomials at Chebyshev points. 

Here we briefly describe the algorithm for efficient generation of
tensor-product Chebyshev coefficients. The univariate Chebyshev points 
of degree $L$ in the interval $[0,1]$ are given by:
\[
t_{j,L} = \left(1-\cos\left(j\pi /L \right)\right)/2, \quad j=0,\ldots,L.
\] 
The Chebyshev coefficients of any tensor product polynomial $f$ of bidegree no
higher than $(L_1,L_2)$ can then be generated by the following 
procedure \cite{trefethen:spectral}:
\begin{itemize}
 \item Construct a matrix $\mathbf{f} =
(f(t_{i,L_1},t_{j,L_2}))_{i=0,j=0}^{L_1,L_2}$ of values of the function $f$ at
the tensor-product Chebyshev points,
 \item Extend $\mathbf{f}$ to its even counterpart $\hat{\mathbf{f}}:$
 \begin{align*}
 \hat{f}_{i,j} & = f_{i,j}, \ \ i = 0,\ldots,L_1, \ j = 0,\ldots,L_2, \\
 \hat{f}_{L_1+i,j} & = f_{L_1-i,j}, \ \ i = 1,\ldots,L_1-1, \ j = 0,\ldots,L_2,
\\
 \hat{f}_{i,L_2+j} & = f_{L_1-i,j}, \ \ i = 0,\ldots,L_1, \ j = 1,\ldots,L_2-1,
\\
 \hat{f}_{L_1+i,L_2+j} & = f_{L_1-i,L_2-j}, \ \ i = 1,\ldots,L_1-1, \ j =
1,\ldots,L_2-1.
 \end{align*}
 \item perform a bivariate fast Fourier transform (FFT) to get
$\tilde{\mathbf{f}} = \text{FFT}(\hat{\mathbf{f}})$,
 \item extract the first $(L_1+1,L_2+1)$ coefficients of $\tilde{\mathbf{f}}$ to
get $\mathbf{g}=(\tilde{f}_{i,j})_{i,j=0}^{L_1,L_2}.$ The 
matrix $\mathbf{g}$ then contains the tensor product Chebyshev coefficients of
$f.$
\end{itemize}

The algorithm for approximate implicitization proceeds by applying the above
procedure to the functions
$\{w^d(\beta_k\circ\mathbf{p})\}_{k=1}^M$ and
$\{-h^2\alpha_{l}\}_{l=1}^{L_1L_2},$ and arranging the coefficients in matrices
$\mathbf{D}_q$ and $\mathbf{D}_\lambda$ according to the definition
(\ref{eq:dmatrix}). The efficiency of the method is due to it being based on
point sampling and FFT. Moreover, the sampling can be done entirely in parallel
making the method highly suitable for implementation on heterogeneous
architectures.

\section{Numerical results}

In this section we present an example of the method along with both computation
times and estimations for the convergence rates. In order to generate reliable
data, we have chosen a degree six family of lines which has a rational
envelope. We can thereby use a parametrization of the envelope to compute the
algebraic error of the approximations. The family of lines is pictured in Figure
\ref{fig:example_lines}, along with the envelope approximations up to the exact
implicitization at degree six. For these approximations we take $k_1 =
\max(0,dn_1-2\deg_s(h))$ and $k_2 = \max(0,dn_2-2\deg_t(h)),$ since this is also 
the minimum needed for the exact solution.

It can be seen that with increased degree the approximations converge quickly.
It is possible that with higher degrees, extra branches may appear in the
region of interest. For example, the approximation of degree five has an extra
branch close to the envelope curve. However, such artifacts could be avoided 
using a suitable collection of low-degree approximations (see 
\cite{schulz:envelope} for an adaptive algorithm).
\begin{figure} 
\begin{center}
 \caption{Approximations of the envelope of a family of lines for degrees
$d=1,\ldots,6.$}\label{fig:example_lines}
\begin{tabular}{ccc}
  \label{fig:app1}\includegraphics[width=0.27\textwidth]{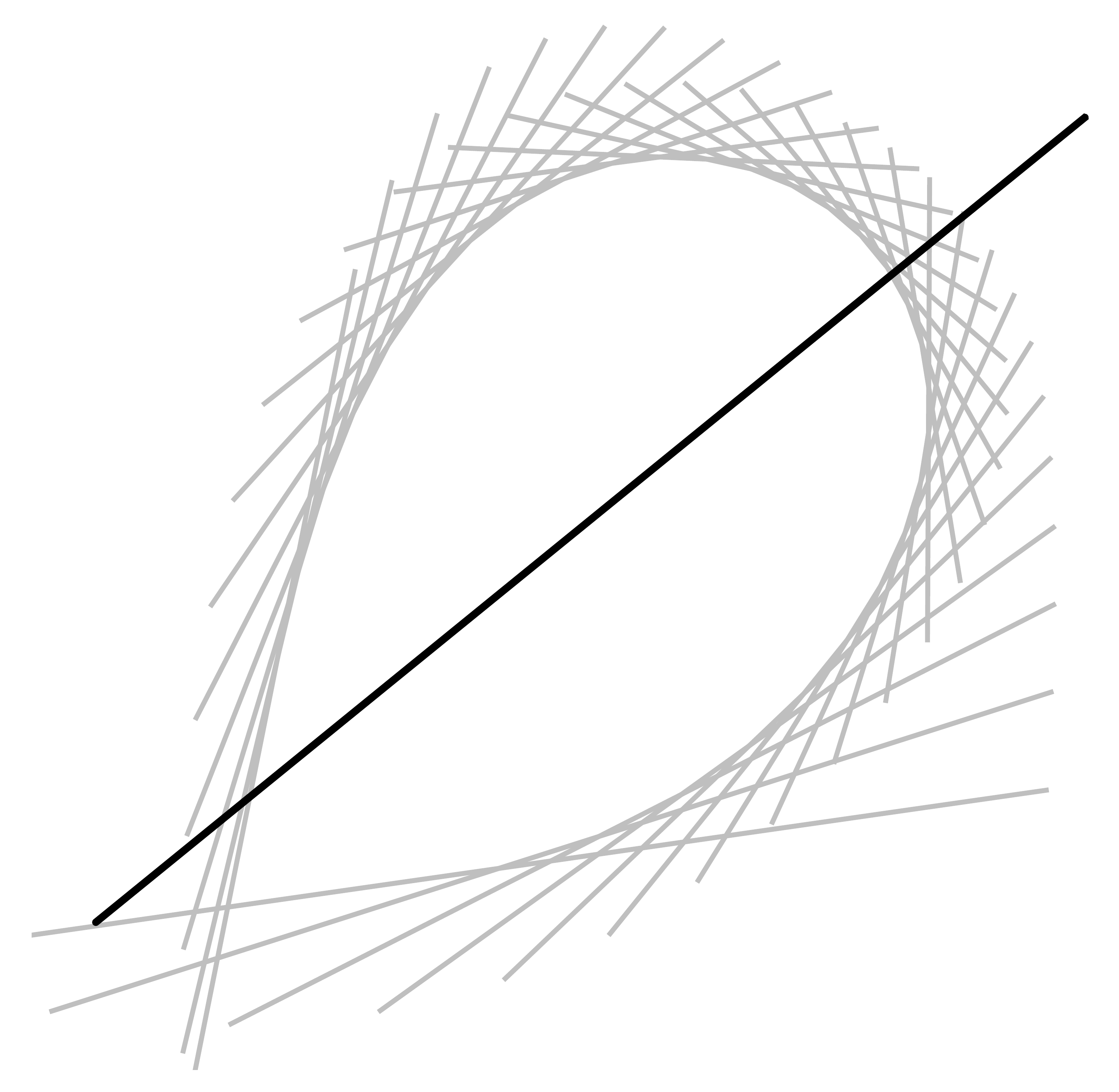}   
& \label{fig:app2}\includegraphics[width=0.27\textwidth]{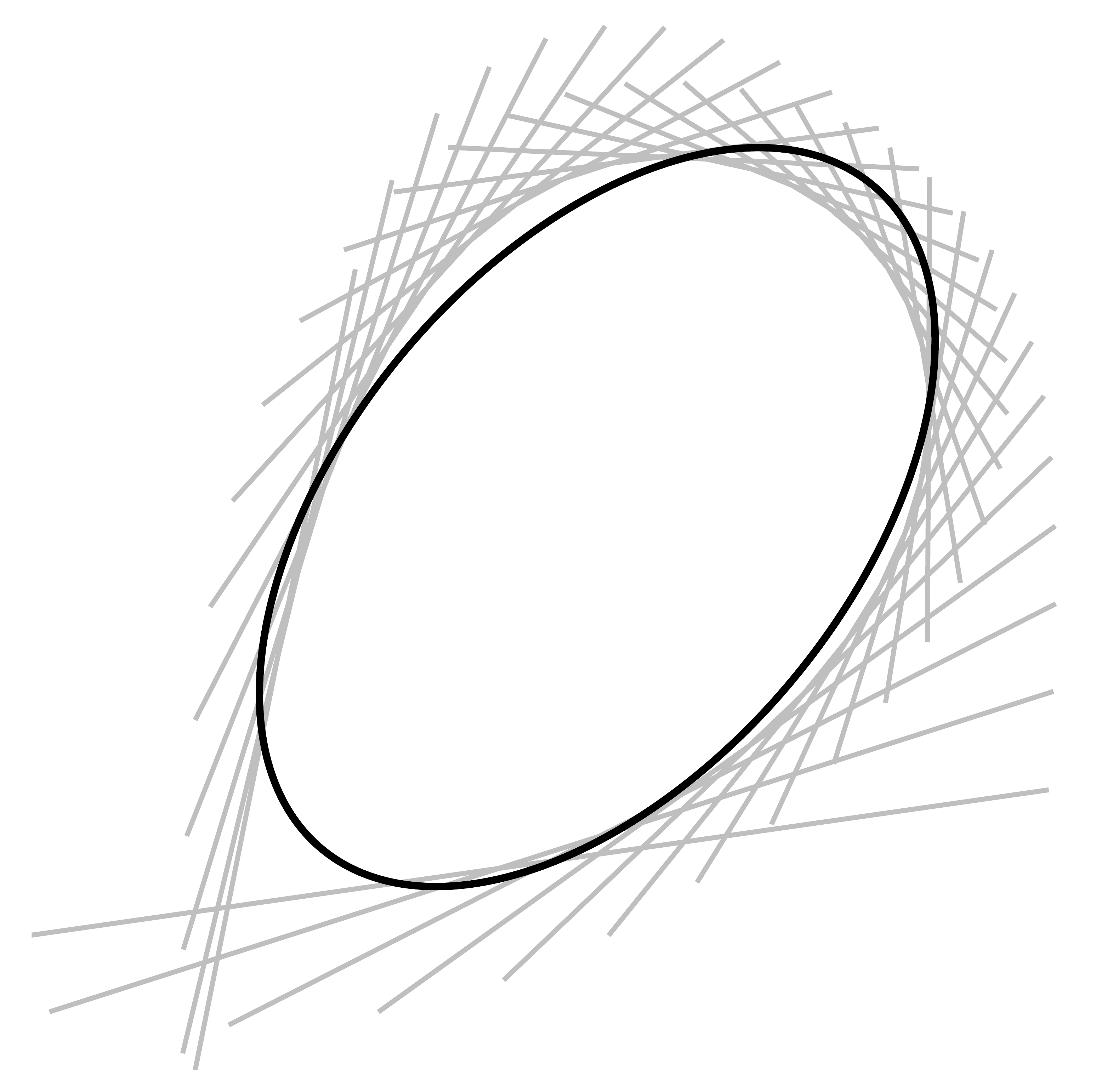}    
& \label{fig:app3}\includegraphics[width=0.27\textwidth]{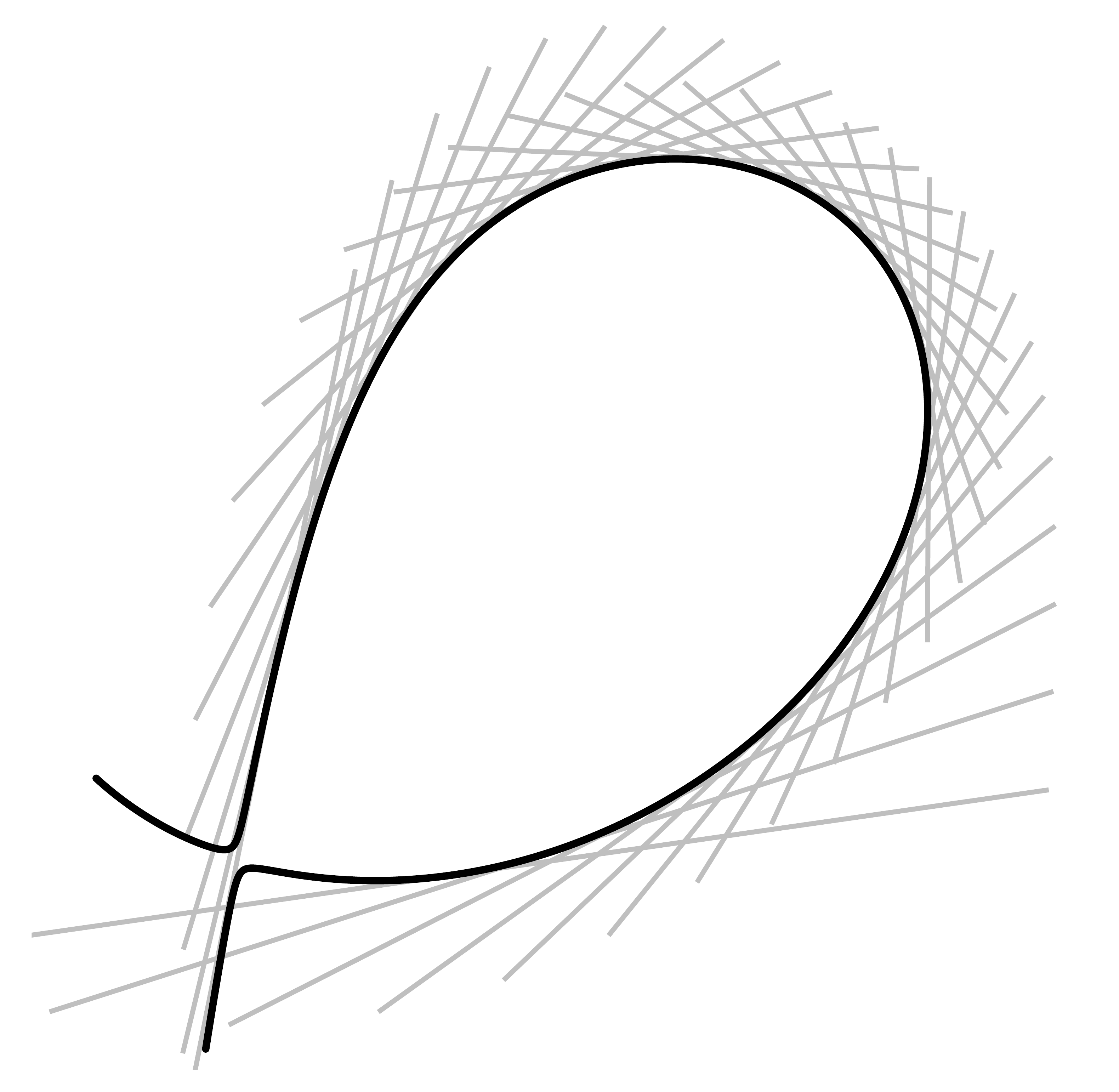}   \\
  $d=1$  & $d=2$ & $d=3$ \\
  \label{fig:app4}\includegraphics[width=0.27\textwidth]{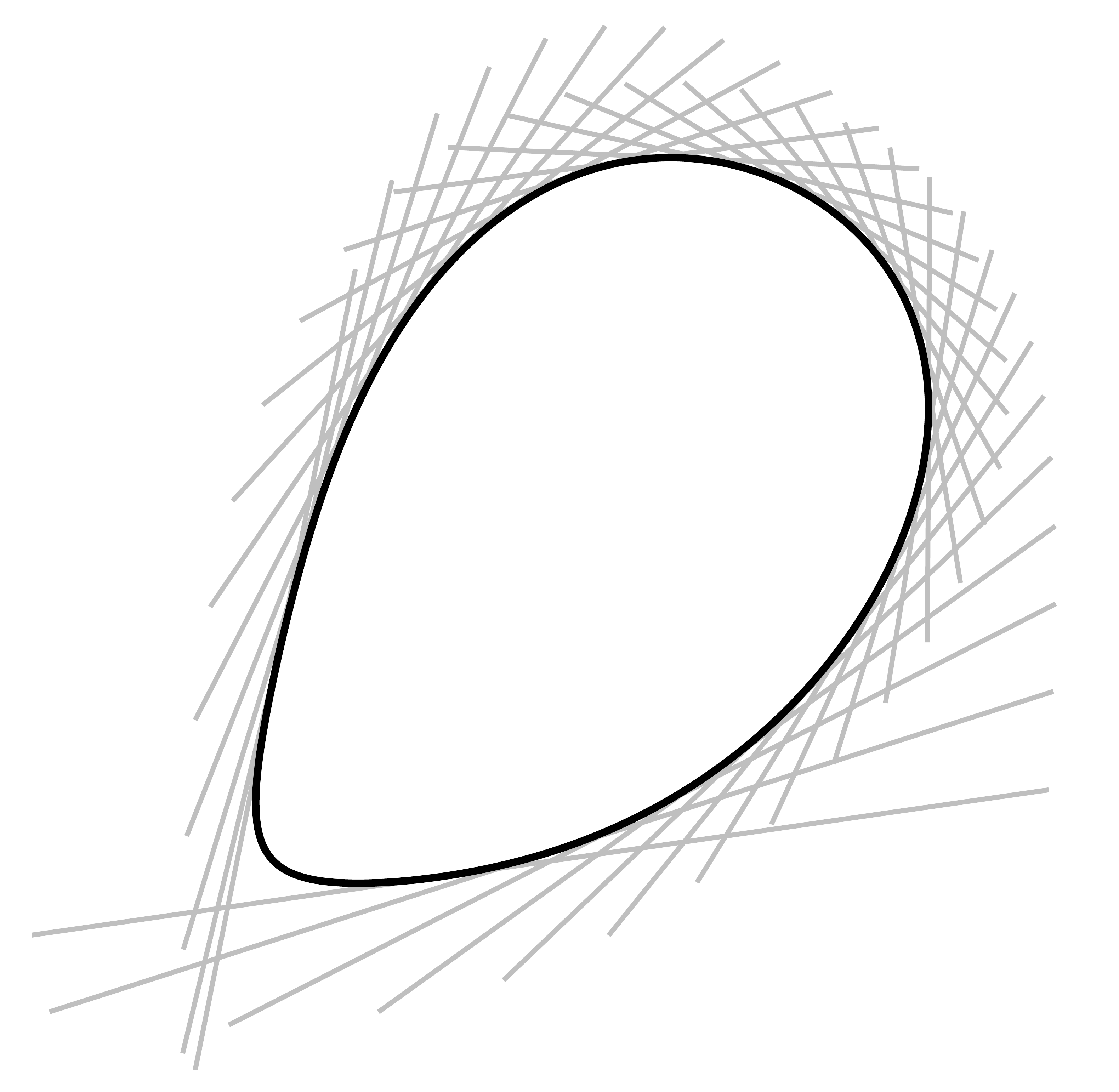}   
& \label{fig:app5}\includegraphics[width=0.27\textwidth]{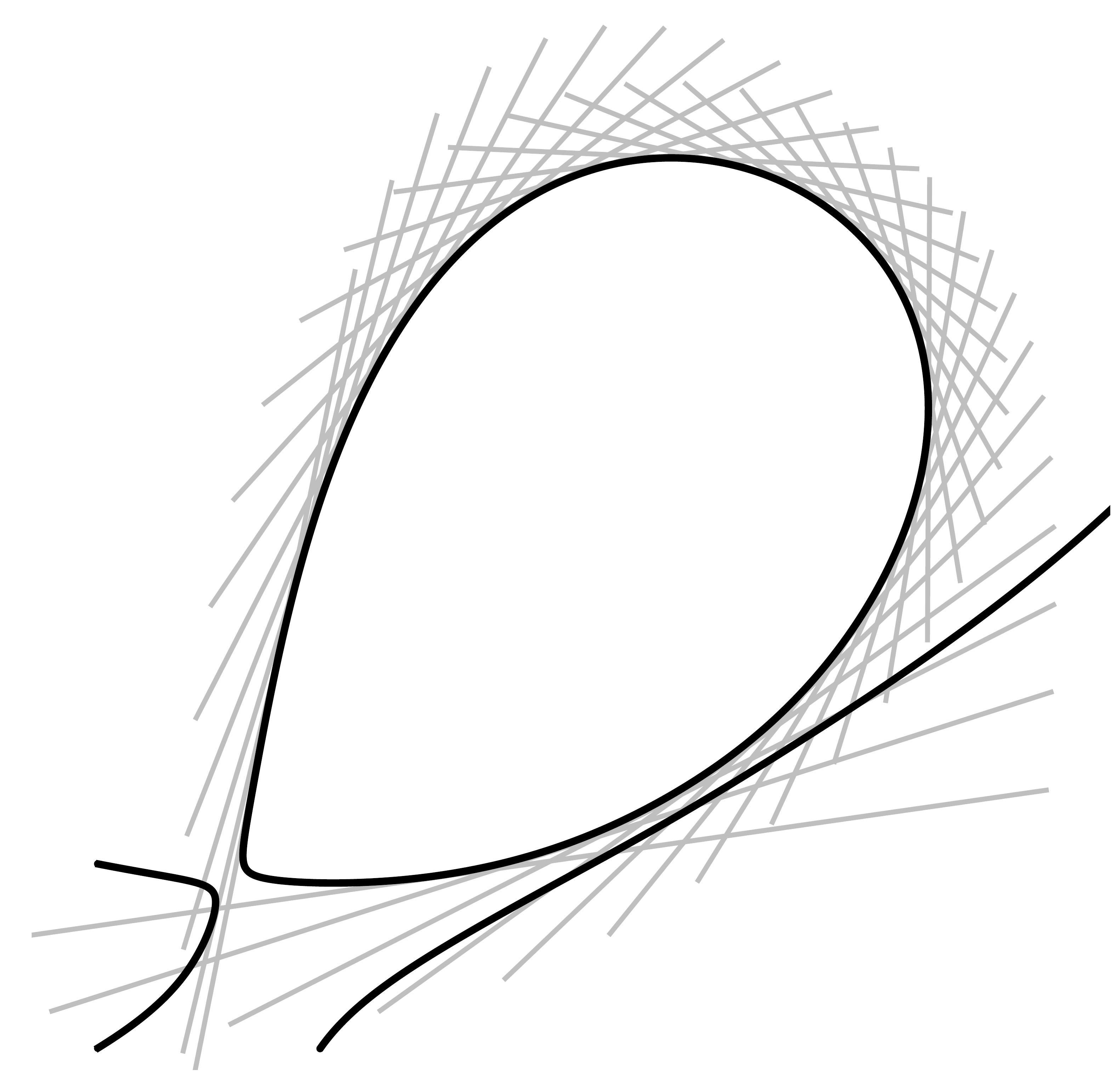}    
& \label{fig:app6}\includegraphics[width=0.27\textwidth]{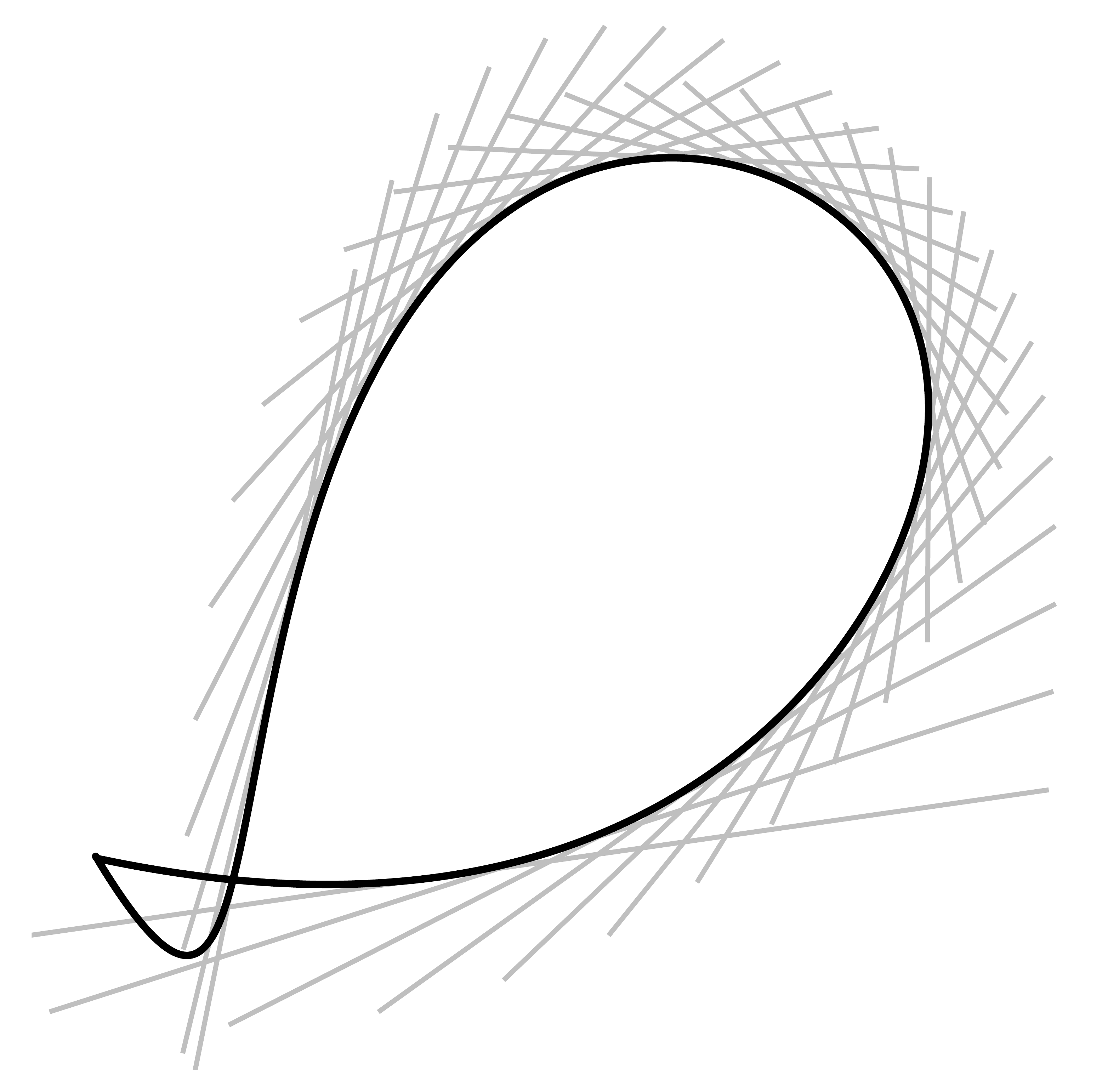}   \\
  $d=4$  & $d=5$ & $d=6$ \\
\end{tabular}
\end{center}
\end{figure}

In Table \ref{tab:times} we show the computation times for the above
approximations. The algorithm has been implemented in the Python 
programming language using the NumPy library for the built in FFT and 
singular value decomposition (SVD) algorithms. The results are computed on a
3.4Ghz Intel Core i7-2600 with 8GB RAM.

\begin{table} 
\begin{center}
 \caption{Computation times and number of matrix coefficients for the
examples in Fig. \ref{fig:example_lines}.}\label{tab:times}
 \begin{tabular}{|l||c|c|c|c|c|c|c|c|}
 \hline
 Degree $d$           & 1      & 2      & 3      & 4      & 5      & 6      \\
\hline 
 \# coefficients      & 196    & 975    & 2964   & 7000   & 14136  & 25641  \\
\hline
 Time (s)             & 0.02   & 0.04   & 0.11   & 0.23   & 0.45   & 0.80   \\
\hline 
 \end{tabular}
\end{center}
\end{table}
Instead of increasing the polynomial degree $d,$ one may also improve the
quality of the approximations by subdivision; the envelope is then approximated
by a piecewise implicit representation. It is thus of interest to see how the
approximation improves as the region $\Omega=I\times J$ is reduced. 

Consider a region $\Omega_i = I_i\times J_i,$ of diameter $2^{-i}$ centered on a
point $(s_0,t_0)$ in 
\[
\mathcal{H} = \{(s,t)\in \Omega:h(s,t)=0\}. 
\]
For an approximation $q_{d,i}$ of degree $d$ to over the region $\Omega_i,$ we
define the maximum algebraic error to be 
\[
\epsilon_{d,i} = \max_{(s,t)\in\mathcal{H}\cap\Omega_i}
|q_{d,i}(\mathbf{p}(s,t))|,
\]
where the coefficients $\mathbf{c}_{q_{d,i}}$ of $q_{d,i}$ have been
renormalized to ${\Vert\mathbf{c}_{q_{d,i}}\Vert=1,}$ in order to give meaningful
results. Given two approximations $q_{d,i},$ and $q_{d,i+1},$ on subsequent
subdivision regions $\Omega_i$ and $\Omega_{i+1},$ we define the convergence rate
to be $r_{d,i} = \log_2(\epsilon_{d,i}/\epsilon_{d,i+1}).$ Table \ref{tab:error}
shows
values of $\epsilon_{d,i}$ and $r_{d,i}$ for four successive subdivisions of the
example in Figure \ref{fig:example_lines} and degrees $d$ up to four. Values of
$\epsilon_{d,i}$ below machine precision have been omitted.

\begin{table}
\begin{center}
 \caption{Maximum algebraic error $\epsilon_{d,i},$ of the approximations of the
example in Fig.\ref{fig:example_lines}, together with approximate convergence
rates
$r_{d,i}.$}\label{tab:error}
\begin{tabular}{c c|c|c|c|c|c|c|c|c|}
\cline{3-10}
& & \multicolumn{8}{|c|}{Implicit Degree $d$} \\ \cline{3-10}
& & \multicolumn{2}{|c|}{1} & \multicolumn{2}{|c|}{2} & \multicolumn{2}{|c|}{3}
& \multicolumn{2}{|c|}{4} \\ \cline{3-10}
& & $\epsilon_{d,i}$ & $r_{d,i}$ & $\epsilon_{d,i}$ & $r_{d,i}$ &
$\epsilon_{d,i}$ & $r_{d,i}$ & $\epsilon_{d,i}$ & $r_i$  \\ \cline{1-10}
\multicolumn{1}{|c}{}{\multirow{5}{*}{ Diameter $2^{-i}$}} & 
\multicolumn{1}{|c|}{1}   
& 1.69e-1 & -     & 6.23e-3  & -     & 1.16e-4  & -     & 3.96e-6  & -      \\
\cline{2-10}
\multicolumn{1}{|c}{} &  \multicolumn{1}{|c|}{1/2} 
& 1.67e-1 & 0.02  & 3.46e-4  & 4.170 & 2.62e-6  & 5.467 & 2.66e-10 & 13.86  \\
\cline{2-10}
\multicolumn{1}{|c}{} &  \multicolumn{1}{|c|}{1/4} 
& 3.04e-2 & 2.458 & 1.52e-5  & 4.511 & 1.50e-9  & 10.77 & 1.34e-14 & 14.27  \\
\cline{2-10}
\multicolumn{1}{|c}{} &  \multicolumn{1}{|c|}{1/8} 
& 6.52e-3 & 2.223 & 5.02e-7  & 4.915 & 2.87e-12 & 9.028 & n/a      & n/a    \\
\cline{2-10}
\multicolumn{1}{|c}{} &  \multicolumn{1}{|c|}{1/16} 
& 1.41e-3 & 2.213 & 1.58e-8  & 4.989 & 5.63e-15 & 8.993 & n/a      & n/a    \\
\cline{2-10}\cline{1-10}
  \end{tabular}
\end{center}
\end{table}

As can be seen from Table \ref{tab:error}, the error $\epsilon_{d,i}$ decreases
both with increased degree and increased levels of subdivision. The values of
$r_{d,i},$ suggest that the convergence rates for $d=1,2,3$ and $4$ are
approximately two, five, nine and 14 respectively. This corresponds directly to
the number of degrees of freedom in approximating with lines, conics, cubics and
quartics and is hence as high a convergence as we can expect, supporting our choices 
for the degrees $(k_1,k_2).$ The results in Table \ref{tab:error} are typical of 
rational examples we have tested.

It should be noted that in general, envelope curves are not rational. Thus, this
example, whilst showing that high convergence rates are attainable, cannot
conclude that this is always the case. However, from studying additional 
examples, our experience shows that convergence behaviour is good in
the general setting.

\section{Conclusion}

We have presented a new implementation of approximate implicitization of
envelope curves using Chebyshev polynomials. We have detailed the computation
times and convergence behaviour of a specific example, thereby demonstrating the
feasibility of our approach. This paper also motivates theoretical work on
convergence rates as a direction for future research.
\\[3ex]{\bf Acknowledgements}
The research leading to these results has received funding from the
European Community's Seventh Framework Programme FP7/2007-2013
under grant agreement n$^\circ$ PITN-GA-2008-214584 (SAGA), and from
the Research Council of Norway (IS-TOPP). It was
also supported by the Doctoral Program ``Computational Mathematics''
(W1214) at the Johannes Kepler University of Linz.

\bibliographystyle{spmpsci}
\bibliography{references}

\end{document}